\documentclass{article}
\usepackage{amssymb}
\usepackage{amsfonts}
\usepackage{graphicx}
\usepackage{amsmath}
\usepackage[sans]{dsfont}

\newtheorem{theorem}{Theorem}

\newtheorem{lemma}{Lemma}

\newtheorem{proposition}{Proposition}
\newtheorem{remark}{Remark}

\newenvironment{proof}[1][Proof]{\textbf{#1.} }{\ \rule{0.5em}{0.5em}}

\begin{document}

\date{}
\title{On the construction of $m$-step methods for FDEs}
\author{Lidia Aceto\thanks{%
Dipartimento di Matematica, Universit\`{a} di Pisa, Italy, lidia.aceto@unipi.it}, \and %
Cecilia Magherini\thanks{%
Dipartimento di Matematica, Universit\`{a} di Pisa, Italy, cecilia.magherini@unipi.it},
\and Paolo Novati \thanks{%
Dipartimento di Matematica, Universit\`{a} di Padova, Italy, novati@math.unipd.it}}
\maketitle

\begin{abstract}
In this paper we consider the numerical solution of Fractional Differential
Equations by means of $m$-step recursions. The construction of such formulas
can be obtained in many ways. Here we study a technique based on the
rational approximation of the generating functions of Fractional Backward
Differentiation Formulas (FBDFs). Accurate approximations allow to define
methods which simulate the theoretical properties of the underlying FBDF
with important computational advantages. Numerical experiments are presented.

\smallskip

\end{abstract}


\section{Introduction}

This paper deals with the solution of Fractional Differential Equations
(FDEs) of the type%
\begin{equation}
_{t_{0}}D_{t}^{\alpha }y(t)=g(t,y(t)),\quad t_{0}<t\leq T,\quad 0<\alpha <1,
\label{fde}
\end{equation}%
where $_{t_{0}}D_{t}^{\alpha }$ denotes the Caputo's fractional derivative
operator (see e.g. \cite{Pod} for an overview) defined as%
\begin{equation}
_{t_{0}}D_{t}^{\alpha }y(t)=\frac{1}{\Gamma (1-\alpha )}\int_{t_{0}}^{t}%
\frac{y^{\prime }(u)}{(t-u)^{\alpha }}du,  \label{cd}
\end{equation}%
in which $\Gamma $ denotes the Gamma function. As well known, the use of the
Caputo's definition for the fractional derivative allows to treat the
initial conditions at $t_{0}$ for FDEs in the same manner as for integer
order differential equations. Setting $y(t_{0})=y_{0}$ the solution of (\ref%
{fde}) exists and is unique under the hypothesis that $g$ is continuous and
fulfils a Lipschitz condition with respect to the second variable (see e.g.
\cite{DN} for a proof).

As for the integer order case $\alpha =1$, a classical approach for solving (%
\ref{fde}) is based on the discretization of the fractional derivative (\ref%
{cd}), which generalizes the well known Grunwald-Letnikov discretization
(see \cite[\S 2.2]{Pod}), leading to the so-called Fractional Backward
Differentiation Formulas (FBDFs) introduced in \cite{Lub}. Taking a uniform
mesh $t_{0}, t_{1}, \dots, t_{N}=T$ of the time domain with stepsize $%
h=(T-t_{0})/N,$ FBDFs are based on the full-term recursion%
\begin{equation}
\sum\nolimits_{j=0}^{n}\omega _{n-j}^{(p)}y_{j}=h^{\alpha
}g(t_{n},y_{n}),\quad p\leq n\leq N,  \label{rec}
\end{equation}%
where $y_{j}\approx y(t_{j})$ and $\omega _{n-j}^{(p)}$ 
are the Taylor coefficients of the generating function%
\begin{eqnarray}
\omega _{p}^{(\alpha )}(\zeta ) &=&\left( a_{0}+a_{1}\zeta +...+a_{p}\zeta
^{p}\right) ^{\alpha }  \label{genf} \\
&=&\sum\nolimits_{i=0}^{\infty }\omega _{i}^{(p)}\zeta ^{i},\quad \text{for}%
\quad 1\leq p\leq 6,  \label{tay}
\end{eqnarray}%
being $\left\{ a_{0},a_{1}, \dots, a_{p}\right\} $ the coefficients of the
underlying BDF. In \cite{Lub} it is also shown that the order $p$ of the BDF
is preserved.

We remember that for this kind of equations, there is generally an intrinsic
lack of regularity of the solution in a neighborhood of the starting point,
that is, depending on the function $g$, one may have $y(t)\sim
(t-t_{0})^{\alpha }$ as $t\rightarrow t_{0}$. For this reason, in order to
preserve the theoretical order $p$ of the numerical method, formula (\ref%
{rec}) is generally corrected as%
\begin{equation}
\sum\nolimits_{j=0}^{M}w_{n,j}y_{j}+\sum\nolimits_{j=0}^{n}\omega
_{n-j}^{(p)}y_{j}=h^{\alpha }g(t_{n},y_{n}),  \label{fbdfc}
\end{equation}%
where the sum $\sum\nolimits_{j=0}^{M}w_{n,j}y_{j}$ is the so-called
starting term, in which $M$ and the weights $w_{n,j}$ depend on $\alpha $
and $p$ (see \cite[Chapter 6]{BV} for a discussion).

Denoting by $\Pi _{m}$ the set of polynomials of degree not exceeding $m$,
our basic idea is to design methods based on rational approximations of (\ref%
{genf}), i.e.,
\begin{equation}
R_{m}(\zeta )\approx \omega _{p}^{(\alpha )}(\zeta ),\quad R_{m}(\zeta )=%
\frac{p_{m}(\zeta )}{q_{m}(\zeta )},\quad p_{m},q_{m}\in \Pi _{m}.
\label{gff}
\end{equation}%
Writing $p_{m}(\zeta )=\sum\nolimits_{j=0}^{m}\alpha _{j}\zeta ^{j}$ and $%
q_{m}(\zeta )=\sum\nolimits_{j=0}^{m}\beta _{j}\zeta ^{j}$, the above
approximation naturally leads to implicit $m$-step recursions of the type%
\begin{equation}
\sum\nolimits_{j=n-m}^{n}\alpha _{n-j}y_{j}=h^{\alpha
}\sum\nolimits_{j=n-m}^{n}\beta _{n-j}g(t_{j},y_{j}),\quad n\geq m.
\label{ks}
\end{equation}%
While the order of the FBDF is lost, we shall see that if the approximation (%
\ref{gff}) is rather accurate, then (\ref{ks}) is able to produce reliable
approximations to the solution. Starting from the initial data $%
y(t_{0})=y_{0}$, the first $m-1$ approximations $y_{1},\ldots,y_{m-1}$ can be
generated by the underlying FBDFs or even considering lower degree rational
approximations.

A formula of type (\ref{ks}) generalizes in some sense the methods based on
the Short Memory Principle in which the truncated Taylor expansion of (\ref%
{genf}) is considered (see \cite[\S 8.3]{Pod} for some examples).
Computationally, the advantages are noticeable, especially in terms of
memory saving whenever (\ref{fde}) arises from the semi-discretization of
fractional partial differential equations. We remark moreover that since the
initial approximations are used only at the beginning of the process, there
is no need to use a starting term to preserve the theoretical order as for
standard full-recursion multistep formulas. Moreover, as remarked in \cite%
{DFFW}, in particular when $\alpha \neq 1/2$, the use of a starting formula
as in (\ref{fbdfc}), that theoretically should ensure the order of the FBDF,
in practice may introduce substantial errors, causing unreliable numerical
solutions. For high-order formulas, this is due to the severe
ill-conditioning of the Vandermonde type systems one has to solve at each
integration step to generate the weights $w_{n,j}$ of the starting term. We
also remark that in a typical application $\alpha $, $y_{0}$ and possibly
also the function $g$ may be only known up to a certain accuracy (see \cite%
{DN} for a discussion), so that one may only be interested in having a
rather good approximation of the true solution.

For the construction of formulas of type (\ref{ks}), in this paper we
present a technique based on the rational approximation of the fractional
derivative operator (cf. \cite{Nov}). After considering a BDF discretization
of order $p$ of the first derivative operator, which can be represented by a
$N\times N$ triangular banded Toeplitz matrix $A_{p}$, we approximate
Caputo's fractional differential operator $_{t_{0}}D_{t}^{\alpha }$ by
calculating $A_{p}^{\alpha }$. This computation is performed by means of the
contour integral approximation, which leads to a rational approximation of $%
A_{p}^{\alpha }$ whose coefficients can be used to define (\ref{ks}). This
technique is based on the fact that the first column of $A_{p}^{\alpha } $
contains the first $N$ coefficients of the Taylor expansion of $\omega
_{p}^{(\alpha )}(\zeta )$, so exploiting the equivalence between the
approximation of $A_{p}^{\alpha }$ and $\omega _{p}^{(\alpha )}(\zeta )$.

The outline of the paper is the following. As in \cite{fgs}, the contour
integral is evaluated by means of the Gauss-Jacobi rule in Section \ref{sec2}%
. An error analysis of this approach is outlined in Section \ref{sec3},
together with some numerical experiments that confirm its effectiveness. In
Section \ref{sec5} we investigate the reliability of this approach for the
solution of fractional differential equations. In particular, we present
some results concerning the consistency and the linear stability. Finally,
in Section \ref{sec6} we consider the results of the method when applied to
the discretization of two well-known models of fractional diffusion.

\section{The approximation of the fractional derivative operator}

\label{sec2}

Denoting by $a_{0}, a_{1}, \dots, a_{p}$ the $p+1$ coefficients of a
Backward Differentiation Formula (BDF) of order $p$, with $1\leq p\leq 6$,
which discretizes the derivative operator (see \cite[Chapter III.1]{HNW} for
a background), we consider lower triangular banded Toeplitz matrices of the
type%
\begin{equation}
A_{p}=\left(
\begin{array}{ccccc}
a_{0} & 0 &  &  & 0 \\
\vdots & a_{0} & 0 &  &  \\
a_{p} & \vdots & \ddots & 0 &  \\
0 & \ddots &  & \ddots & 0 \\
& 0 & a_{p} & \cdots & a_{0}%
\end{array}%
\right) \in \mathbb{R}^{(N+1)\times (N+1)}.  \label{ap}
\end{equation}%
In this setting, $A_{p}^{\alpha }e_{1}$, $e_{1}=(1,0,\dots,0)^{T}$, contains
the whole set of coefficients of the corresponding FBDF for approximating
the solution of (\ref{fde}) in $t_{0},t_{1},\dots,t_{N}$, that is
\begin{equation}\label{omjp}
e_{j+1}^{T}A_{p}^{\alpha }e_{1}=\omega _{j}^{(p)},\quad 0\leq j\leq N,
\end{equation}
(cf. (\ref{tay})). The constraint $p\leq 6$ is due to the fact that BDFs are
not zero-stable for $p>6.$

From the theory of matrix functions (see \cite{Hi} for a survey), we know
that the fractional power of matrix can be written as a contour integral%
\begin{equation*}
A^{\alpha }=\frac{A}{2\pi i}\int_{\Gamma }z^{\alpha -1}(zI-A)^{-1}dz,
\end{equation*}%
where $\Gamma $ is a suitable closed contour enclosing the spectrum of $A$, $%
\sigma (A)$, in its interior. The following known result (see, e.g., \cite%
{BHM}) expresses $A^{\alpha }$ in terms of a real integral.

\begin{proposition}
\label{pro1}Let $A\in \mathbb{R}^{N\times N}$ be such that $\sigma
(A)\subset \mathbb{C}\backslash \left( -\infty ,0\right] $. For $0<\alpha <1$
the following representation holds%
\begin{equation}
A^{\alpha }=\frac{A\sin (\alpha \pi )}{\alpha \pi }\int_{0}^{\infty }(\rho
^{1/\alpha }I+A)^{-1}d\rho .  \label{rei}
\end{equation}
\end{proposition}

Of course the above result holds also in our case since $\sigma
(A_{p})=\{a_{0}\}$ and $a_{0}>0$ for each $1\leq p\leq 6.$ At this point,
for each suitable change of variable for $\rho $, a $k$-point quadrature
rule for the computation of the integral in (\ref{rei}) yields a rational
approximation of the type%
\begin{equation}
A_{p}^{\alpha }\approx A_{p}\widetilde{R}_{k}(A_{p}):=A_{p}\sum%
\nolimits_{j=1}^{k}\gamma _{j}(\eta _{j}I+A_{p})^{-1},  \label{rapp}
\end{equation}%
where the coefficients $\gamma _{j}$ and $\eta _{j}$ depend on the
substitution and the quadrature formula. This technique has been used in
\cite{Nov}, where the author applies the Gauss-Legendre rule to (\ref{rei})
after the substitution%
\begin{equation*}
\rho =a_{0}^{\alpha }\left( \cos \theta \right) ^{-\alpha /(1-\alpha )}\sin
\theta ,
\end{equation*}%
which generalizes the one presented in \cite{HHT} for the case $\alpha =1/2.$

In order to remove the dependence of $\alpha $ inside the integral we
consider the change of variable
\begin{equation}
\rho ^{1/\alpha }=\tau \frac{1-t}{1+t},\qquad \tau >0,  \label{js}
\end{equation}%
yielding
\begin{eqnarray*}
&&\frac{1}{\alpha }\int_{0}^{\infty }(\rho ^{1/\alpha }I+A_{p})^{-1}d\rho \\
&=&2\int_{-1}^{1}\left( \tau \frac{1-t}{1+t}\right) ^{\alpha -1}\left( \tau
\frac{1-t}{1+t}I+A_{p}\right) ^{-1}\frac{\tau }{\left( 1+t\right) ^{2}}dt \\
&=&2\tau ^{\alpha }\int_{-1}^{1}\left( 1-t\right) ^{\alpha -1}\left(
1+t\right) ^{-\alpha }\left( \tau \left( 1-t\right) I+\left( 1+t\right)
A_{p}\right) ^{-1}dt,
\end{eqnarray*}%
and hence%
\begin{equation}
A_{p}^{\alpha }=\frac{2\sin (\alpha \pi )\tau ^{\alpha }}{\pi }%
A_{p}\int_{-1}^{1}\left( 1-t\right) ^{\alpha -1}\left( 1+t\right) ^{-\alpha
}\left( \tau \left( 1-t\right) I+\left( 1+t\right) A_p\right) ^{-1}dt.
\label{nint}
\end{equation}%
The above formula naturally leads to the use of a $k$-point Gauss-Jacobi
rule for the approximation of $A_{p}^{\alpha }e_{1}$ and hence to a rational
approximation (\ref{rapp}).

The following result can be proved by direct computation.

\begin{proposition}
\label{p1}Let $A_{p}\in \mathbb{R}^{N\times N}$ be a matrix of type (\ref{ap}%
), and let $\overline{A}_{p}=\frac{1}{a_{0}}A_{p}$. Then the components of $%
(\xi I+\overline{A}_{p})^{-1}e_{1}$, $\xi \neq -1$, are given by
\begin{eqnarray*}
\upsilon _{1}^{(p)}(\xi ) &=&\frac{1}{\xi +1},  \\
\upsilon _{j}^{(p)}(\xi ) &=&\frac{c_{2,j}^{(p)}}{\left( \xi +1\right) ^{2}}%
+\ldots+\frac{c_{j,j}^{(p)}}{\left( \xi +1\right) ^{j}},\quad 2\leq j\leq N,
\end{eqnarray*}%
where the coefficients $c_{i,j}^{(p)}$ depend on the order $p$. For $p=1$ we
simply have $\left\{ a_{0},a_{1}\right\} =\left\{ 1,-1\right\} $, and hence%
\begin{equation*}
\upsilon _{j}^{(1)}(\xi )=\frac{1}{\left( \xi +1\right) ^{j}},\quad 1\leq
j\leq N.
\end{equation*}
\end{proposition}

The above proposition shows that the components of
\begin{equation*}
\left( \tau \left( 1-t\right) I+\left( 1+t\right) A_{p}\right) ^{-1}e_{1}
\end{equation*}%
are analytic functions in a suitable open set containing $[-1,1]$ in its
interior, since they are sum of functions of the type%
\begin{equation} \label{fl}
\frac{\left( 1+t\right) ^{l-1}}{\left( \tau \left( 1-t\right) +a_0\left(
1+t\right) \right) ^{l}},\quad l\geq 1,
\end{equation}%
whose pole lies outside $[-1,1]$ for $\tau >0$ (recall that $a_0>0$ for $1\leq p \leq 6$). 
In this sense, the lack of
regularity of the integrand in (\ref{nint}) due to the presence of $\alpha,$
is completely absorbed by the Jacobi weight function so that the
Gauss-Jacobi rule yields a very efficient tool for the computation of $%
A_{p}^{\alpha }$.

Increasing $k$ the approximation (\ref{rapp}) can be used to approximate the
whole set of coefficients of the FBDFs. We remark that the computation of
the vectors $(\eta _{j}I+A_{p})^{-1}e_{1}$ does not constitute a problem
because of the structure of $A_{p}$ (see (\ref{ap})). We also point out that
since our aim is to construct reliable formulas of type (\ref{ks}) we
actually do not need to evaluate (\ref{rapp}). Indeed we just need to know
the scalars $\gamma _{j}$ and $\eta _{j}$, and then, using an algorithm to
transform partial fractions to polynomial quotient, we obtain the
approximation%
\begin{equation}
z^{\alpha }\approx z\widetilde{R}_{k}(z)=z\frac{\widetilde{p}_{k-1}(z)}{%
\widetilde{q}_{k}(z)},\quad z=a_{0}+a_{1}\zeta +\ldots+a_{p}\zeta ^{p},
\label{ptqt}
\end{equation}%
where $\widetilde{p}_{k-1}\in \Pi _{k-1}$ and $\widetilde{q}_{k}\in \Pi
_{k}. $ This finally leads to the approximation (\ref{gff}) with
\begin{eqnarray}
p_{m}(\zeta ) &=&(a_{0}+a_{1}\zeta +\ldots+a_{p}\zeta ^{p})\,\widetilde{p}%
_{k-1}(a_{0}+a_{1}\zeta +\ldots+a_{p}\zeta ^{p}),  \label{pm} \\
q_{m}(\zeta ) &=&\widetilde{q}_{k}(a_{0}+a_{1}\zeta +\ldots+a_{p}\zeta ^{p}),
\label{qm}
\end{eqnarray}%
in which $m=kp$. We remark that whenever the procedure for the definition of
the coefficients $\gamma _{j}$ and $\eta _{j}$ has been set for a given $%
\alpha, $ one can compute the corresponding coefficients in the $m$-step
formula (\ref{ks}) once and for all.

\section{Theoretical error analysis}

\label{sec3}

Denoting by $J_{k}(A_{p})$ the result of the Gauss-Jacobi rule for the
approximation of
$$J(A_{p})=\int_{-1}^{1}\left( 1-t\right) ^{\alpha -1}\left( 1+t\right)
^{-\alpha }\left( \tau \left( 1-t\right) I+\left( 1+t\right) A_{p}\right)
^{-1}dt, 
$$
by (\ref{nint}) the corresponding approximation to $A_{p}^{\alpha }$ is
given by%
\begin{equation}
A_{p}^{\alpha }\approx A_{p}\widetilde{R}_{k}(A_{p}),\qquad \widetilde{R}%
_{k}(A_{p})=\frac{2\sin (\alpha \pi )\tau ^{\alpha }}{\pi }J_{k}(A_{p}).
\label{rtk}
\end{equation}%
In this section we analyze the error term componentwise, that is, (see (\ref{omjp}),
\begin{eqnarray}
E_{j} &:=&\omega _{j}^{(p)}-e_{j+1}^{T}A_{p}\widetilde{R}%
_{k}(A_{p})e_{1}=e_{j+1}^{T}A_{p}\left( A_{p}^{\alpha -1}-\widetilde{R}%
_{k}(A_{p})\right) e_{1}  \label{Ej0} \\
&=&\frac{2\sin (\alpha \pi )\tau ^{\alpha }}{\pi }e_{j+1}^{T}A_{p}\left(
J(A_{p})-J_{k}(A_{p})\right) e_{1},\quad 0\le j\leq N,  \notag
\end{eqnarray}%
which is the error in the computation of the $j$-th coefficient of the
Taylor expansion of $\omega _{p}^{(\alpha )}(\zeta )$. Numerically one
observes that the quality of the approximation tends to deteriorate when the
dimension of the problem $N$ grows. In this sense we are particularly
interested in observing the dependence of the error term on $j$ and $k$ for $%
j\gg k$ and to find a strategy to define the parameter $\tau $ of the
substitution (\ref{js}) in this situation. As we shall see in the remainder
of the paper, this parameter plays a crucial role for the quality of the
approximation.

We restrict our analysis to the case of $p=1$ for which $a_0=-a_1=1.$ 
In this situation, defining
the vector%
\begin{equation}
r:=\left( J(A_{1})-J_{k}(A_{1})\right) e_{1},  \label{rj}
\end{equation}%
we have that
\begin{equation*}
E_{j}=\frac{2\sin (\alpha \pi )\tau ^{\alpha }}{\pi }e_{j+1}^{T}A_{1}r,
\end{equation*}%
and therefore
\begin{equation}
\left\vert E_{j}\right\vert \leq \frac{2\sin (\alpha \pi )\tau ^{\alpha }}{%
\pi }\left( \left\vert r_{j}\right\vert +\left\vert r_{j-1}\right\vert
\right) .  \label{Ej}
\end{equation}%
The analysis thus reduces to the study of the components of the vector (\ref%
{rj}). By Proposition \ref{p1},  see also (\ref{fl}), the $j$-th component of the vector 
$$\left(\tau \left( 1-t\right) I+\left( 1+t\right) A_{1}\right) ^{-1}e_{1}$$ 
is given by%
\begin{equation}
f_{j}(t)=\frac{\left( 1+t\right) ^{j-1}}{\left( \tau (1-t)+1+t\right) ^{j}},
\label{fj}
\end{equation}%
so that $r_{j}$ is the error term of the $k$-point Gauss-Jacobi formula
applied to the computation of%
\begin{equation}
\int_{-1}^{1}\left( 1-t\right) ^{\alpha -1}\left( 1+t\right) ^{-\alpha
}f_{j}(t)dt.  \label{fjj}
\end{equation}

We start with the following known result, \cite{Hil}.

\begin{theorem}
\label{t1}The error term of the $k$-point Gauss-Jacobi formula applied to
the computation of%
\begin{equation*}
\int\nolimits_{-1}^{1}(1-t)^{\alpha _{1}}(1+t)^{\alpha _{2}} g(t) dt,\quad
\alpha _{1},\alpha _{2}>-1,\quad g\in C^{2k}([-1,1]),
\end{equation*}%
is given by%
\begin{equation*}
C_{k,\alpha _{1},\alpha _{2}}g^{(2k)}(\xi ),\quad -1<\xi <1,
\end{equation*}%
where%
\begin{equation}
C_{k,\alpha _{1},\alpha _{2}}:=\frac{\Gamma (k+\alpha _{1}+1)\Gamma
(k+\alpha _{2}+1)\Gamma (k+\alpha _{1}+\alpha _{2}+1)k!}{(2k+\alpha
_{1}+\alpha _{2}+1)\left[ \Gamma (2k+\alpha _{1}+\alpha _{2}+1)\right]
^{2}(2k)!}2^{2k+\alpha _{1}+\alpha _{2}+1}.  \label{cab}
\end{equation}
\end{theorem}

\begin{lemma}
\label{l1}Let $0<\alpha <1$ and $C_{k,\alpha }:=C_{k,\alpha -1,-\alpha }$.
Then
\begin{equation*}
C_{k,\alpha }\sim \frac{\pi 2^{1-2k}}{(2k)!}.
\end{equation*}
\end{lemma}

\begin{proof}
By (\ref{cab}) we easily obtain%
\begin{equation*}
C_{k,\alpha }=\frac{\Gamma (k+\alpha )\Gamma (k-\alpha +1)\left[ \Gamma (k)%
\right] ^{2}}{\left[ \Gamma (2k)\right] ^{2}(2k)!}2^{2k-1}.
\end{equation*}%
Using the Legendre formula
\begin{equation*}
\Gamma \left( k+\frac{1}{2}\right) \Gamma (k)=\sqrt{\pi }\Gamma (2k)2^{1-2k},
\end{equation*}%
we have that%
\begin{equation*}
C_{k,1/2}=\frac{\pi 2^{1-2k}}{(2k)!}.
\end{equation*}%
Moreover, since for $a,b\in (0,1)$%
\begin{equation*}
k^{b-a}\frac{\Gamma (k+a)}{\Gamma (k+b)}=1+O\left( \frac{1}{k}\right) ,
\end{equation*}%
we have that $\Gamma(k+\alpha)\Gamma(k-\alpha+1) = \left[\Gamma(k+\frac{1}{2})\right]^2 (1+O(k^{-1}))$ and,
consequently, $C_{k,\alpha }\rightarrow C_{k,1/2}$ as $k\rightarrow \infty $.
\end{proof}

\begin{remark}\label{rpade}
\label{rempade} If we set $\tau =1$ in (\ref{js}) we obtain $r_{j}=0$ for
each $j=1,2,\ldots ,2k$ since $f_{j}\in \Pi _{j-1},$ see (\ref{fj}). From (%
\ref{ptqt}), with $p=1,$ and (\ref{rtk}) one therefore gets
$$(1-\zeta )^{\alpha -1}-\widetilde{R}_{k}(1-\zeta )=(1-\zeta )^{\alpha -1}-%
\frac{\widetilde{p}_{k-1}(1-\zeta )}{\widetilde{q}_{k}(1-\zeta )}=O(\zeta
^{2k}), $$
so that $\widetilde{R}_{k}(1-\zeta )$ is the $(k-1,k)$ Pad\'{e} approximant
of $(1-\zeta )^{\alpha -1}$ with expansion point $\zeta =0.$ More generally,
if $\tau \in (0,1]$ then the resulting rational approximation coincides with
the same Pad\'{e} approximant with expansion point $\zeta =1-\tau \ $(cf.
\cite[Lemma~4.4]{fgs}).
\end{remark}

Numerically, it is quite clear that the best results are obtained for $\tau $
strictly less than 1, so that in what follows we always assume to work with $%
\tau \in (0,1)$. Indeed, in this situation we are able to approximate the
Taylor coefficients with a more uniform distribution of the error with
respect to $j.$ By Theorem~\ref{t1}, we need to bound $\left\vert
f_{j}^{(2k)}(t)\right\vert $ in the interval $[-1,1]$ in order to bound the
error term $E_{j}.$ We start with the following result.

\begin{proposition}
\label{p3} Let $0<\tau <1\ $and%
\begin{equation}
a:=\frac{1+\tau}{1-\tau }.  \label{atau}
\end{equation}%
For each $j$ and $k$
$$
\max_{\lbrack -1,1]}\left\vert f_{j}^{(2k)}(t)\right\vert \leq \left(
2k\right) !\frac{\sqrt{a}}{\left( \sqrt{a}-1\right) ^{2k+2}}\left( \frac{a+1%
}{2\sqrt{a}}\right) ^{j}. 
$$
\end{proposition}

\begin{proof}
The function $f_{j}$ can be written as%
\begin{equation*}
f_{j}(t)=\frac{\left( 1+t\right) ^{j-1}}{\left( a+t\right) ^{j}}\left( \frac{%
a+1}{2}\right) ^{j},\quad a>1.
\end{equation*}%
Using the Cauchy integral formula we have%
\begin{equation}
f_{j}^{(2k)}(t)=\frac{\left( 2k\right) !}{2\pi i}\int_{\Gamma }\frac{f_{j}(w)%
}{\left( w-t\right) ^{2k+1}}dw,  \label{ci}
\end{equation}%
where $\Gamma $ is a contour surrounding $t$ but not the pole $-a<-1$. We
take $\Gamma $ as the circle centered at the origin and radius $\rho $ such
that $1<\rho <a$, that is, we use the substitution $w=\rho e^{i\theta }$. We
obtain%
\begin{equation}
\max_{\lbrack -1,1]}\left\vert f_{j}^{(2k)}(t)\right\vert \leq \left(
2k\right) !\frac{\rho }{\left( \rho -1\right) ^{2k+1}}\max_{[0,2\pi
]}\left\vert f_{j}(\rho e^{i\theta })\right\vert .  \label{bnd2}
\end{equation}%
Taking $\rho =\sqrt{a}$ we have%
\begin{eqnarray*}
\left\vert f_{j}(\rho e^{i\theta })\right\vert &=&\left\vert \frac{1}{1+\rho
e^{i\theta }}\right\vert \left\vert \frac{1+\rho e^{i\theta }}{a+\rho
e^{i\theta }}\right\vert ^{j}\left( \frac{a+1}{2}\right) ^{j} \\
&=&\left\vert \frac{1}{1+\sqrt{a}e^{i\theta }}\right\vert \left\vert \frac{1+%
\sqrt{a}e^{i\theta }}{a+\sqrt{a}e^{i\theta }}\right\vert ^{j}\left( \frac{a+1%
}{2}\right) ^{j} \\
&=&\left\vert \frac{1}{1+\sqrt{a}e^{i\theta }}\right\vert \left( \frac{a+1}{2%
\sqrt{a}}\right) ^{j} \\
&\leq &\frac{1}{\left( \sqrt{a}-1\right) }\left( \frac{a+1}{2\sqrt{a}}%
\right) ^{j}.
\end{eqnarray*}%
By (\ref{bnd2}) we immediately achieve the result.
\end{proof}

The above result is rather accurate only for small values of $j$. Since $%
f_{j}(t)$ is growing in the interval $[-1,1]$, below we consider contours $%
\Gamma $ in (\ref{ci}) which are dependent on $t$, in order to balance this
effect.

\begin{proposition}
\label{p2}Let $0<\tau <1\ $ and $a$ as in (\ref{atau}). For $j\geq 2k+2$
\begin{equation}
\left\vert f_{j}^{(2k)}(t)\right\vert \leq \left( 2k\right) !\frac{\sqrt{a+1}%
+\sqrt{2}}{a-1}\frac{\left( 1+t\right) ^{\frac{j-1}{2}-k}}{(a+t)^{\frac{j}{2}%
+k}}\left( \frac{a+1}{2}\right) ^{j},\quad t\in \lbrack -1,1].  \label{bndd}
\end{equation}%
Moreover%
\begin{equation}
\max_{\lbrack -1,1]}\left\vert f_{j}^{(2k)}(t)\right\vert \leq \left(
2k\right) !\left( \sqrt{a+1}+\sqrt{2}\right) \frac{\left( \frac{j-2k-1}{4k+1}%
\right) ^{\frac{j-1}{2}-k}}{\left( \frac{j+2k}{4k+1}\right) ^{\frac{j}{2}+k}}%
\left( \frac{a+1}{2}\right) ^{j}\left( a-1\right) ^{-2k-\frac{3}{2}}
\label{b1}
\end{equation}%
for $j$ such that
\begin{equation}
a\leq \frac{j+6k+1}{j-2k-1},  \label{as}
\end{equation}%
and%
\begin{equation}
\max_{\lbrack -1,1]}\left\vert f_{j}^{(2k)}(t)\right\vert \leq \left(
2k\right) !\frac{\sqrt{a+1}+\sqrt{2}}{a-1}\frac{\left( a+1\right) ^{\frac{j}{%
2}-k}}{2^{\frac{j+1}{2}+k}}  \label{b2}
\end{equation}%
otherwise.
\end{proposition}

\begin{proof}
In (\ref{ci}) we take $\Gamma $ as the circle centered at $t$ and of radius $%
\rho $ such that $1+t<\rho <t+a$, that is, we use the substitution $w=t+\rho
e^{i\theta }$. We obtain%
\begin{equation}
\left\vert f_{j}^{(2k)}(t)\right\vert \leq \left( 2k\right) !\frac{1}{\rho
^{2k}}\max_{[0,2\pi ]}\left\vert f_{j}(t+\rho e^{i\theta })\right\vert .
\label{bnd}
\end{equation}%
For $t>-1$ we define $\rho =\sqrt{(t+a)(1+t)}$, so that%
\begin{eqnarray*}
\left\vert f_{j}(t+\rho e^{i\theta })\right\vert &=&\left\vert \frac{1}{%
1+t+\rho e^{i\theta }}\right\vert \left\vert \frac{1+t+\rho e^{i\theta }}{%
a+t+\rho e^{i\theta }}\right\vert ^{j}\left( \frac{a+1}{2}\right) ^{j} \\
&=&\left\vert \frac{1}{1+t+\sqrt{(t+a)(1+t)}e^{i\theta }}\right\vert \times
\\
&&\left\vert \frac{1+t+\sqrt{(t+a)(1+t)}e^{i\theta }}{a+t+\sqrt{(t+a)(1+t)}%
e^{i\theta }}\right\vert ^{j}\left( \frac{a+1}{2}\right) ^{j} \\
&=&\left\vert \frac{1}{1+t+\sqrt{(t+a)(1+t)}e^{i\theta }}\right\vert \left(
\frac{1+t}{a+t}\right) ^{\frac{j}{2}}\left( \frac{a+1}{2}\right) ^{j},
\end{eqnarray*}%
and hence%
\begin{eqnarray*}
\max_{\lbrack 0,2\pi ]}\left\vert f_{j}(t+\rho e^{i\theta })\right\vert &=&%
\frac{1}{\sqrt{(t+a)(1+t)}-(1+t)}\left( \frac{1+t}{a+t}\right) ^{\frac{j}{2}%
}\left( \frac{a+1}{2}\right) ^{j} \\
&=&\frac{\sqrt{(t+a)}+\sqrt{(1+t)}}{\sqrt{(1+t)}(a-1)}\left( \frac{1+t}{a+t}%
\right) ^{\frac{j}{2}}\left( \frac{a+1}{2}\right) ^{j}.
\end{eqnarray*}%
By (\ref{bnd}) we obtain the bound (\ref{bndd}) for each $j$ and $k$, for $%
t>-1$. By continuity, (\ref{bndd}) holds for $t\in \lbrack -1,1]$ if $j\geq
2k+2$.\newline

Now, we observe that the maximum with respect to $t$ of the function%
\begin{equation*}
\frac{\left( 1+t\right) ^{\frac{j-1}{2}-k}}{(a+t)^{\frac{j}{2}+k}}
\end{equation*}%
is attained at
$$t^{\ast }=\frac{a(j-2k-1)-(j+2k)}{4k+1}\geq -1. $$
Moreover $t^{\ast }\leq 1$ for $a$ verifying (\ref{as}).
Substituting $t^{\ast }$ in (\ref{bndd}) leads to (\ref{b1}). If $t^{\ast
}>1 $ then the maximum of (\ref{bndd}) is reached at $t=1$ and hence we
obtain (\ref{b2})$.$
\end{proof}

In order to derive (\ref{b1}) and (\ref{b2}) we have assumed $\tau $ to be a
priori fixed. Now, using these bounds, we look for the value of $\tau $
which minimize $\left\Vert f_{j}^{(2k)}\right\Vert $ for a given $j$. For $%
j\geq 2k+3$, the minimization of the term
\begin{equation*}
\left( a+1\right) ^{j}\left( a-1\right) ^{-2k-\frac{3}{2}}
\end{equation*}%
in (\ref{b1}) leads to%
\begin{equation*}
a^{(1)}=\frac{2j+4k+3}{2j-4k-3},
\end{equation*}%
which satisfies (\ref{as}). Consequently, by (\ref{atau}) we obtain%
\begin{equation}
\tau ^{(1)}=\frac{4k+3}{2j}.  \label{tau}
\end{equation}%
On the other side, for the same $j\geq 2k+3,$ the minimization of the term%
\begin{equation*}
\frac{\sqrt{a+1}+\sqrt{2}}{a-1}\left( a+1\right) ^{\frac{j}{2}-k}
\end{equation*}%
in (\ref{b2}) leads to a value%
\begin{equation*}
a^{\ast }\leq \frac{j-2k+2}{j-2k-2},
\end{equation*}%
which also satisfies (\ref{as}) and hence does not fulfill the requirement
of (\ref{b2}). For $a>a^{\ast }$ the bound (\ref{b2}) is growing with $a$
and consequently its minimum is attained just for%
\begin{equation*}
a^{(2)}=\frac{j+6k+1}{j-2k-1}.
\end{equation*}%
Using this value in (\ref{b2}) leads to a bound that is coarser than the one
obtained by replacing $a^{(1)}$ in (\ref{b1}). For this reason, with respect
to our estimates, $\tau $ given by (\ref{tau}) represents the optimal value
for the computation of (\ref{fjj}) and consequently of $\omega _{j}^{(p)}$.

The following theorem summarizes the results obtained. The proof follows
straightfully from (\ref{Ej}), Theorem \ref{t1}, Lemma \ref{l1}, and
Propositions \ref{p3}, \ref{p2}.

\begin{theorem}
Let $0<\tau <1\ $and
\begin{equation*}
a:=\frac{1+\tau}{1-\tau }.
\end{equation*}%
Then%
$$\left\vert E_{j}\right\vert \leq 2^{3-2k}\sin (\alpha \pi )\tau ^{\alpha
}\Psi (a,j,k). 
$$
where%
\begin{equation*}
\Psi (a,j,k):=\left\{
\begin{array}{c}
\frac{\sqrt{a}}{\left( \sqrt{a}-1\right) ^{2k+2}}\left( \frac{a+1}{2\sqrt{a}}%
\right) ^{j}, \quad j\leq 2k+1, \\
\left( \sqrt{a+1}+\sqrt{2}\right) \frac{\left( \frac{j-2k-1}{4k+1}\right) ^{%
\frac{j-1}{2}-k}}{\left( \frac{j+2k}{4k+1}\right) ^{\frac{j}{2}+k}}\left(
\frac{a+1}{2}\right) ^{j}\left( a-1\right) ^{-2k-\frac{3}{2}}, \\
2k+2\leq j\leq \frac{6k+1+a(2k+1)}{a-1}, \\
\frac{\sqrt{a+1}+\sqrt{2}}{a-1}\frac{\left( a+1\right) ^{\frac{j}{2}-k}}{2^{%
\frac{j+1}{2}+k}},\quad j\geq \max \left( 2k+2,\frac{6k+1+a(2k+1)}{a-1}%
\right).%
\end{array}%
\right.
\end{equation*}%
For $\tau =\tau ^{(1)}$ the corresponding expression of $\Psi (a,j,k)$ is
minimized for $j\geq 2k+3$. 
\end{theorem}

\subsection{Numerical experiments}

\label{sec4}

As already mentioned, the aim of the whole analysis was to have indications
about the choice of the parameter $\tau $ with respect to the degree $k$ of
the formula and the dimension of the problem $N$. Unfortunately, the
definition of $\tau $ as in (\ref{tau}) depends on $j$, while we need a
value which is as good as possible for each $1\leq j\leq N$. In this sense,
the idea, confirmed by the forthcoming experiments, is to use $\tau ^{(1)}$
with $j=N/2$, that is, focusing the attention on the middle of the interval $%
[0,N]$. This leads to a choice of $\tau $ around the value $4k/N$. We remark
that the previous analysis was restricted to the case of $p=1,$ because of
the difficulties in dealing with the functions $f_{j}$ for $p>1$ (cf.
Proposition~\ref{p1}).

Numerically, we can proceed as follows. If we define
$$
\tau ^{\ast }=\mbox{arg}\min_{\tau }{\mathcal{E}}(\tau ),\qquad {\mathcal{E}}%
(\tau ):=\left\Vert A_{p}^{\alpha -1}-\widetilde{R}_{k}(A_{p})\right\Vert
_{\infty },  
$$
then, in principle, $\tau ^{\ast }=\tau ^{\ast }(\alpha ,k,N,p).$ However,
the numerical experiments done by using the \texttt{Matlab} optimization
routine \texttt{fminsearch} indicate that the dependence on $\alpha $ is
negligible with respect to the others. In particular, there is numerical
evidence that ${\mathcal{E}}(\tau ^{\ast })\approx {\mathcal{E}}(\hat{\tau})$
where
\begin{equation}
\hat{\tau}=\frac{(7+p)}{2N}k.  \label{tauhat}
\end{equation}%
In Figures~\ref{taub1}-\ref{taub3} we report the values of ${\mathcal{E}}%
(\tau ^{\ast })$ and ${\mathcal{E}}(\hat{\tau})$ for $p=1,3,$ respectively.
We recall that the corresponding sets of coefficients $\left\{
a_{0},a_{1},\ldots ,a_{p}\right\} $ in (\ref{ap}) are given by
\begin{eqnarray*}
p &=&1:\quad \left\{ 1,-1\right\} , \\
p &=&3:\quad \left\{ 11/6,-3,3/2,-1/3\right\} .
\end{eqnarray*}%
As one can see, all the curves are approximatively overlapped. A
\textquotedblleft quasi\textquotedblright\ optimal approximation of $%
A_{p}^{\alpha -1}$ can be therefore obtained by using the very simple
formula in (\ref{tauhat}) for choosing $\tau .$ Moreover, it is important to
remark that such approximations are surely satisfactory even with $k\ll N.$

\begin{figure}[h]
\centerline{\includegraphics[width=10cm,height=5.5cm]{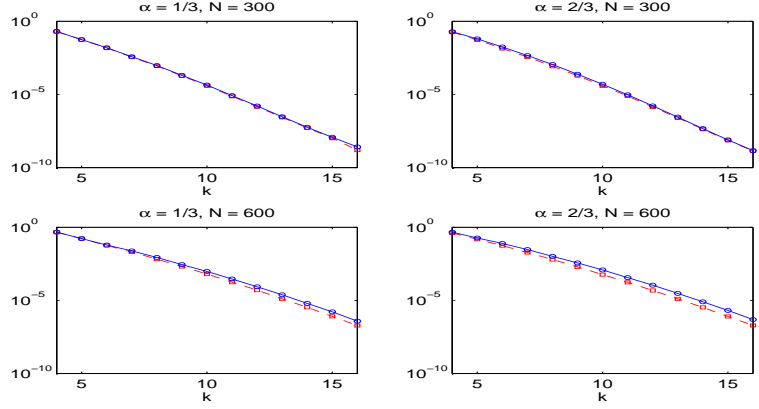}}
\caption{Error behavior of the Gauss-Jacobi rule for the approximation of $%
A_{1}^{\protect\alpha -1}$ for $\protect\tau =\protect\tau ^{\ast }$ (dashed
line) and $\protect\tau =\hat{\protect\tau}=4k/N$ (solid line). }
\label{taub1}
\end{figure}

\begin{figure}[h]
\centerline{\includegraphics[width=10cm,height=5.5cm]{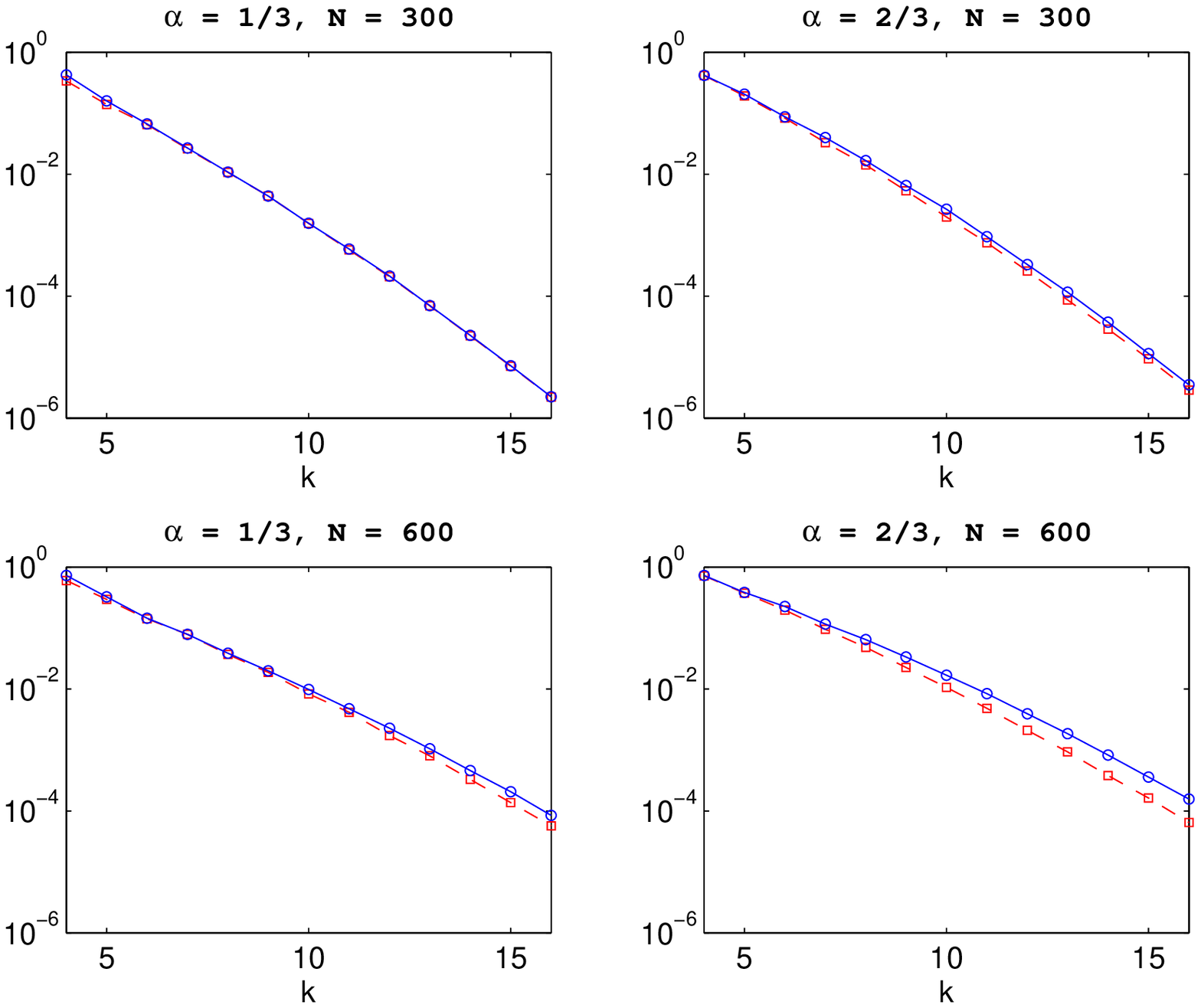}}
\caption{Error behavior of the Gauss-Jacobi rule for the approximation of $%
A_{3}^{\protect\alpha -1}$ for $\protect\tau =\protect\tau ^{\ast }$ (dashed
line) and $\protect\tau =\hat{\protect\tau}=5k/N$ (solid line).}
\label{taub3}
\end{figure}

The previous results are all related to the overall error in the
approximation of $A_{p}^{\alpha }.$ Considering that our final goal is the
use of such approximation for the solution of FDEs, it is important to
inspect also the componentwise error. As an example, in Figure \ref{figcomp}%
, we report such errors, i.e., the values of $E_{j}$ defined in (\ref{Ej0}),
in the case of $N=400,$ $p=1,$ $\tau =\hat{\tau}$ for $\alpha =0.3,0.5,0.7$,
and different values of $k$. We also consider the componentwise errors of
the polynomial approximation of the generating function obtained by
truncating its Taylor series, with memory length equal to 16. Obviously this
is equivalent to approximate with 0 the coefficients $\omega _{i}^{(p)}$ of (%
\ref{tay}), for $i>16$, so that the error is just $\left\vert \omega
_{i}^{(p)}\right\vert $. The competitiveness of the rational approach is
undeniable.

\begin{figure}[tbh]
\centerline{\includegraphics[width=10cm,height=4cm]{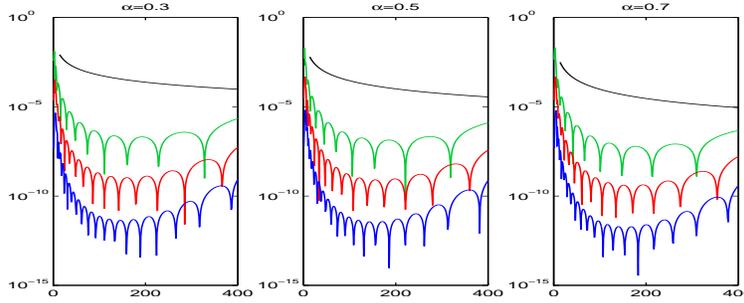}}
\caption{Componentwise error of the Gauss-Jacobi rational approximation with
memory length $k=6,9,12$, and the polynomial approximation with $k=16$.}
\label{figcomp}
\end{figure}

\section{The solution of FDEs}

\label{sec5} In this section we discuss the use of the described
approximation of $A_{1}^{\alpha }$ for getting a $k$-step method that
simulates the FBDF of order 1. The discrete problem provided by the latter
method applied for solving (\ref{fde}) can be written in matrix form as
follows
\begin{equation}
\left( A_{1}^{\alpha }\otimes I_{s}\right) \left( Y-{\mathds{1}}\otimes
y_{0}\right) =h^{\alpha }G(Y),  \label{pdfbdf1}
\end{equation}%
where $s$ is the dimension of the FDE, $I_{s}$ is the identity matrix of
order $s,$ $y_{0}\in \mathbb{R}^{s}$ represents the initial value, $%
h=(T-t_{0})/N$ is the stepsize, ${\mathds{1}}=\left( 1,1,\ldots ,1\right)
^{T}\in \mathbb{R}^{N},$
\begin{equation*}
Y=\left(
\begin{array}{c}
y_{1} \\
y_{2} \\
\vdots \\
y_{N}%
\end{array}%
\right) \approx \left(
\begin{array}{c}
y(t_{1}) \\
y(t_{2}) \\
\vdots \\
y(t_{N})%
\end{array}%
\right) ,\qquad G(Y)=\left(
\begin{array}{c}
g(t_{1},y_{1}) \\
g(t_{2},y_{2}) \\
\vdots \\
g(t_{N},y_{N})%
\end{array}%
\right) \equiv \left(
\begin{array}{c}
g_{1} \\
g_{2} \\
\vdots \\
g_{N}%
\end{array}%
\right) .
\end{equation*}%
As described in Section~\ref{sec2}, the use of a $k$-point Gauss-Jacobi rule
for approximating (\ref{nint}) leads to
$$
A_{1}^{\alpha }\approx \left(
\begin{array}{ccccc}
\beta _{0} & 0 &  &  & 0 \\
\vdots & \beta _{0} & 0 &  &  \\
\beta _{k} & \vdots & \ddots & 0 &  \\
0 & \ddots &  & \ddots & 0 \\
& 0 & \beta _{k} & \cdots & \beta _{0}%
\end{array}%
\right) ^{-1}\left(
\begin{array}{ccccc}
\alpha _{0} & 0 &  &  & 0 \\
\vdots & \alpha _{0} & 0 &  &  \\
\alpha _{k} & \vdots & \ddots & 0 &  \\
0 & \ddots &  & \ddots & 0 \\
& 0 & \alpha _{k} & \cdots & \alpha _{0}%
\end{array}%
\right) \equiv B^{-1}C. 
$$
Here, the coefficients $\left\{ \alpha _{j}\right\} _{j=0}^{k}$ and $\left\{
\beta _{j}\right\} _{j=0}^{k}$ are related to the rational approximation
through the formulas (\ref{ptqt})--(\ref{qm}), with $m=k$ since $p=1,$
\begin{equation}
p_{k}(\zeta )=(1-\zeta )\widetilde{p}_{k-1}(1-\zeta )=\sum_{j=0}^{k}\alpha
_{j}\zeta ^{j},\qquad q_{k}(\zeta )=\widetilde{q}_{k}(1-\zeta
)=\sum_{j=0}^{k}\beta _{j}\zeta ^{j}.  \label{pkqk}
\end{equation}%
If we replace $A_1^{\alpha }$ by $B^{-1}C$ in (\ref{pdfbdf1}) and we
multiply both side of the resulting equation from the left by $B\otimes
I_{s},$ we obtain
\begin{equation}
\left( C\otimes I_{s}\right) Y-C{\mathds{1}}\otimes y_{0}=h^{\alpha
}(B\otimes I_{s})G(Y),  \label{pdCB}
\end{equation}%
where now $Y$ represents the numerical solution provided by the $k$-step
method. In fact, considering that $(C{\mathds{1}})_{n}=0$ for each $%
n=k+1,\ldots ,N,$ since
\begin{equation}
p_{k}(1)=\sum_{j=0}^{k}\alpha _{j}=0,  \label{precons}
\end{equation}%
the discrete problem (\ref{pdCB}) simplifies to
\begin{eqnarray}
\sum_{j=0}^{n-1}\alpha _{j}\left( y_{n-j}-y_{0}\right) &=&h^{\alpha
}\sum_{j=0}^{n-1}\beta _{j}g_{n-j},\qquad n=1,\ldots ,k,  \label{start} \\
\sum_{j=0}^{k}\alpha _{j}y_{n-j} &=&h^{\alpha }\sum_{j=0}^{k}\beta
_{j}g_{n-j},\qquad n=k+1,\ldots ,N.  \label{poi}
\end{eqnarray}%
Indeed, the equations in (\ref{start}) allow to get an approximation of the
solution over the first $k$ meshpoints which are then used as starting
values for the $k$-step recursion in (\ref{poi}).

\begin{remark}\label{remcost}
From (\ref{precons})-(\ref{poi}) follows that the method reproduces exactly
constant solutions, i.e. it is exact if $g(t,y(t))\equiv 0.$
\end{remark}

As it happens in the case of ODEs, a localization of the zeros of the
characteristic polynomials of the $k$-step method in (\ref{pkqk}) is
required in order to study its stability properties. Clearly, such
polynomials depend on the parameter $\tau ,$ i.e. $p_{k}(\zeta )\equiv
p_{k}(\zeta; \tau )$ and $q_{k}(\zeta )\equiv q_{k}(\zeta; \tau )$ since
this dependence occurs in $\widetilde{p}_{k-1},\widetilde{q}_{k}.$ The
method is therefore based on the following rational approximation
\begin{equation}
(1-\zeta )^{\alpha -1}\approx \frac{\widetilde{p}_{k-1}(1-\zeta ;\tau )}{%
\widetilde{q}_{k}(1-\zeta ;\tau )}\equiv \widetilde{R}_{k}(1-\zeta ;\tau ).
\label{rtilde}
\end{equation}%

\begin{theorem}
\label{tipipc} For each $\tau \in (0,1],$ the adjoint of the characteristic
polynomials of the $k$-step method, i.e. $\zeta^kp_k(\zeta^{-1};\tau)$ and $%
\zeta^kq_k(\zeta^{-1};\tau),$ are a Von Neumann and a Schur polynomial,
respectively.
\end{theorem}

\begin{proof}
From Remark~\ref{rpade}, one obtains
\begin{equation}
\widetilde{R}_{k}(1-\zeta ;\tau )=\tau ^{\alpha -1}\widetilde{R}_{k}\left(
\frac{1-\zeta }{\tau };1\right) ,  \label{eqRK}
\end{equation}%
since
\begin{equation*}
\left. \frac{d^{l}}{d\zeta ^{l}}(1-\zeta )^{\alpha -1}\right\vert _{\zeta
=1-\tau }=\left. \tau ^{\alpha -1}\frac{d^{l}}{d\zeta ^{l}}\widetilde{R}%
_{k}\left( \frac{1-\zeta }{\tau };1\right) \right\vert _{\zeta =1-\tau
},\qquad l=0,1,\ldots ,2k-1.
\end{equation*}%
In addition, using the Gauss hypergeometric functions, in \cite[Theorem~4.1]%
{GGZ} it has been proved that
\begin{equation*}
\widetilde{R}_{k}\left( \frac{1-\zeta }{\tau };1\right) =\frac{%
~_{2}F_{1}\left( 1-k,1-\alpha -k;1-2k;(\tau -1+\zeta )/\tau \right) }{%
~_{2}F_{1}\left( -k,\alpha -k;1-2k;(\tau -1+\zeta )/\tau \right) },
\end{equation*}%
or equivalently, by denoting with ${\mathcal{P}}_{r}^{(\gamma ,\beta )}$ the
Jacobi polynomial of degree $r$ and by using \cite[eq.~142, p.\,464]{pbm} and
the symmetry of such polynomials,
\begin{equation}
\widetilde{R}_{k}\left( \frac{1-\zeta }{\tau };1\right) =\tau \frac{(\tau-1+\zeta)
^{k-1}{\mathcal{P}}_{k-1}^{(1-\alpha ,\alpha )}\left( 2\tau /(\tau -1+\zeta
)-1\right) }{(\tau-1+\zeta)^{k}{\mathcal{P}}_{k}^{(\alpha -1,-\alpha )}\left( 2\tau
/(\tau -1+\zeta )-1\right) }.  \label{rtilde2}
\end{equation}%
From (\ref{pkqk}) and (\ref{rtilde})--(\ref{rtilde2}), one therefore gets
\begin{eqnarray*}
p_{k}(\zeta ;\tau ) &=&(1-\zeta )\tau ^{\alpha }(\tau-1+\zeta) ^{k-1}\mathcal{P%
}_{k-1}^{(1-\alpha ,\alpha )}\left( 2\tau /(\tau -1+\zeta )-1\right) , \\
q_{k}(\zeta ;\tau ) &=&(\tau-1+\zeta) ^{k}{\mathcal{P}}_{k}^{(\alpha -1,-\alpha
)}\left( 2\tau /(\tau -1+\zeta )-1\right) .
\end{eqnarray*}%
It follows that, if we denote with $\theta _{i}$ the $i$th root of ${%
\mathcal{P}}_{k-1}^{(1-\alpha ,\alpha )}$ then the roots of $p_{k}(\zeta
;\tau )$ are given by
\begin{equation}
\zeta _{i}=1+\tau \frac{1-\theta _{i}}{1+\theta _{i}}>1,\quad i=1,\ldots
,k-1,\qquad \zeta _{k}=1,  \label{zetap}
\end{equation}%
where the inequality follows from the fact that the roots of the Jacobi
polynomials belong to $(-1,1).$ Similarly, by denoting with $\vartheta _{i}$
the $i$th root of ${\mathcal{P}}_{k}^{(\alpha -1,-\alpha )},$ one deduces
that the roots of $q_{k}(\zeta ;\tau )$ read
\begin{equation}
\zeta _{i}=1+\tau \frac{1-\vartheta _{i}}{1+\vartheta _{i}}>1,\quad
i=1,\ldots ,k.  \label{zetaq}
\end{equation}%
From (\ref{zetap})-(\ref{zetaq}) the statement follows immediately.
\end{proof}

An important consequence of the previous result is that the finite
recurrence scheme is always $0$-stable independently of the stepnumber $k$
and $\tau \in (0,1].$ More precisely, in the case of $g \equiv 0$ the zero
solution of (\ref{poi}) is stable with respect to perturbations of the
initial values.

\subsection{Consistency}

In this section we examine the consistency of the method. While,
theoretically, it is only exact for constant solutions (see Remark~\ref{remcost}), numerically one
observes that the consistency is rather well simulated if $k$ is large
enough. The analysis will also provide some hints about the choice of the
memory length $m$. We restrict our consideration to the case $p=1$ ($m=k$)
but the generalization is immediate.

For a given $y(t)$, the FBDF of order $1$ yields the approximation
\begin{equation*}
_{0}D_{t}^{\alpha }y(t)=\frac{1}{h^{\alpha }}\sum_{j=0}^{N}(-1)^{j}%
\binom{\alpha }{j}\left(y(t-jh)-y(0)\right)+O(h),\quad t=Nh.
\end{equation*}%
Let%
\begin{equation*}
\Delta _{h}^{\alpha }y(t):=\sum_{j=0}^{N}(-1)^{j}\binom{\alpha }{j}%
\left(y(t-jh)-y(0)\right).
\end{equation*}%
Writing a rational approximation of degree $k$ to $\omega _{1}^{(\alpha
)}(\zeta )=\left( 1-\zeta \right) ^{\alpha }$ as%
\begin{equation*}
R_{k}(\zeta )=\sum_{j=0}^{\infty }\gamma _{j}\zeta ^{j},
\end{equation*}%
the corresponding method produces an approximation of the type%
\begin{equation*}
_{0}D_{t}^{\alpha }y(t)\approx \frac{1}{h^{\alpha }}\sum_{j=0}^{N}%
\gamma _{j}\left(y(t-jh)-y(0)\right).
\end{equation*}%
Denoting by%
\begin{equation*}
R_{k,h}^{\alpha }y(t):=\sum_{j=0}^{N}\gamma _{j}\left(y(t-jh)-y(0)\right),
\end{equation*}%
we obtain%
\begin{eqnarray*}
\lefteqn{_{0}D_{t}^{\alpha }y(t) -\frac{1}{h^{\alpha }}R_{k,h}^{\alpha }y(t)=}\\
&&=  ~_{0}D_{t}^{\alpha }y(t)-\frac{1}{h^{\alpha }}\Delta _{h}^{\alpha }y(t)+%
\frac{1}{h^{\alpha }}\Delta _{h}^{\alpha }y(t)-\frac{1}{h^{\alpha }}%
R_{k,h}^{\alpha }y(t)\\
&&=O(h)+\frac{1}{h^{\alpha }}\sum\nolimits_{j=0}^{N}\left[ (-1)^{j}\binom{%
\alpha }{j}-\gamma _{j}\right]\left( y(t-jh)-y(0) \right).
\end{eqnarray*}%
The consistency of the method is ensured if
\begin{equation}
\frac{1}{h^{\alpha }}\sum\nolimits_{j=0}^{N}\left[ (-1)^{j}\binom{\alpha }{j}%
-\gamma _{j}\right] \left(y(t-jh)-y(0)\right)\rightarrow 0  \label{tm_2}
\end{equation}%
as $h\rightarrow 0$ (cf. \cite{GG}). While this cannot be true for a fixed $%
k<\infty $, in what follows we show that numerically, i.e., for $h\geq
h_{0}>0$, the consistency is well simulated if $k$ is large enough and if
the rational approximation to $A_{p}^{\alpha }$ is reliable.

As pointed out in \cite{Lub}, a certain method for FDEs with generating
function $\omega ^{(\alpha )}(\zeta )$ is consistent of order $p$ if%
\begin{equation*}
h^{-\alpha}\omega ^{(\alpha )}(e^{-h})=1+O(h^{p}).
\end{equation*}%
In this setting, in order to understand the numerical consistence of our
method, we consider the above relation by replacing $\omega ^{(\alpha
)}(e^{-h})$ with $\omega_{1}^{(\alpha )}(e^{-h})$ and $R_{k}(e^{-h})$. In
particular, if we set 
\begin{equation}
q_{k}(h)=\log _{h}\left( h^{-\alpha }\left|R_{k}(e^{-h}) -\omega _{1}^{(\alpha
)}(e^{-h})\right| \right) ,  \label{qk2}
\end{equation}%
then we obtain
\begin{eqnarray*}
h^{-\alpha }\left|R_{k}(e^{-h})\right| &\leq& h^{-\alpha }\left|\omega_1^{(\alpha)}(e^{-h})\right| +
h^{-\alpha}\left|R_{k}(e^{-h})-\omega_1^{(\alpha)}(e^{-h})\right|.\\
&=& 1+O(h) + h^{q_{k}(h)}
\end{eqnarray*}
This implies that the consistency of the FBDF of the first order  
is well simulated  as long as $q_k(h)$ is larger than $1.$
In Figure~\ref{cons1_2}, we plot such function for $\alpha =1/2,$ $\tau =1/10$, and different values of $k.$\\
\begin{figure}
\centerline{\includegraphics[width=8cm,height=5cm]{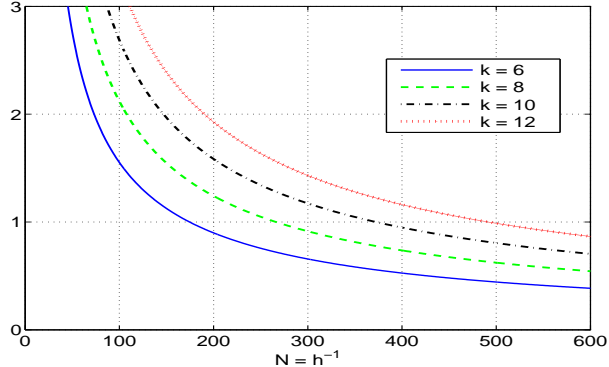}}
\caption{Plot of the function $q_{k}(h)$ in (\protect\ref{qk2}) for 
$\protect%
\alpha =1/2,$ $\protect\tau =1/10$ and different values of $k.$ }
\label{cons1_2}
\end{figure}

The previous experiment does not take care of the perturbation 
introduced in the approximation of the fractional derivative of fractional
powers of the independent variable which may be present in the solution of the 
FDE. In order to control such perturbations, we therefore consider the following second experiment. 
Going back to formula (\ref{tm_2}), we let $N=1/h$ and $y(t)=E_{\alpha }(-t^{\alpha })$
where $E_{\alpha }(x)$ denotes the one-parameter Mittag-Leffler function
(see e.g. \cite[Chapter~1]{Pod})
\begin{equation}\label{mlf}
E_{\alpha }(x)=\sum_{k=0}^{\infty }\frac{x^{k}}{\Gamma (k\alpha +1)}.
\end{equation}%
In Figure \ref{cons2_2}, we then consider the behavior of the function
\begin{equation}
\widetilde{q}_{k}(h)=\log _{h}\left( h^{-\alpha }\left\vert
\sum\nolimits_{j=0}^{N}\left( (-1)^{j}\left(
\begin{array}{c}
\alpha  \\
j%
\end{array}%
\right) -\gamma _{j,k}\right) (y(t_{N-j})-y(0))\right\vert \right)
\label{qkk_2}
\end{equation}
which, similarly to $q_k(h),$ has to be compared with $1.$
The values of $y(t)$ have been
computed using the \texttt{Matlab} function \texttt{mlf} from \cite{Pod2}
that implements the Mittag-Leffler function together with the Schur-Parlett
algorithm.\\
\begin{figure}
\centerline{\includegraphics[width=8cm,height=5cm]{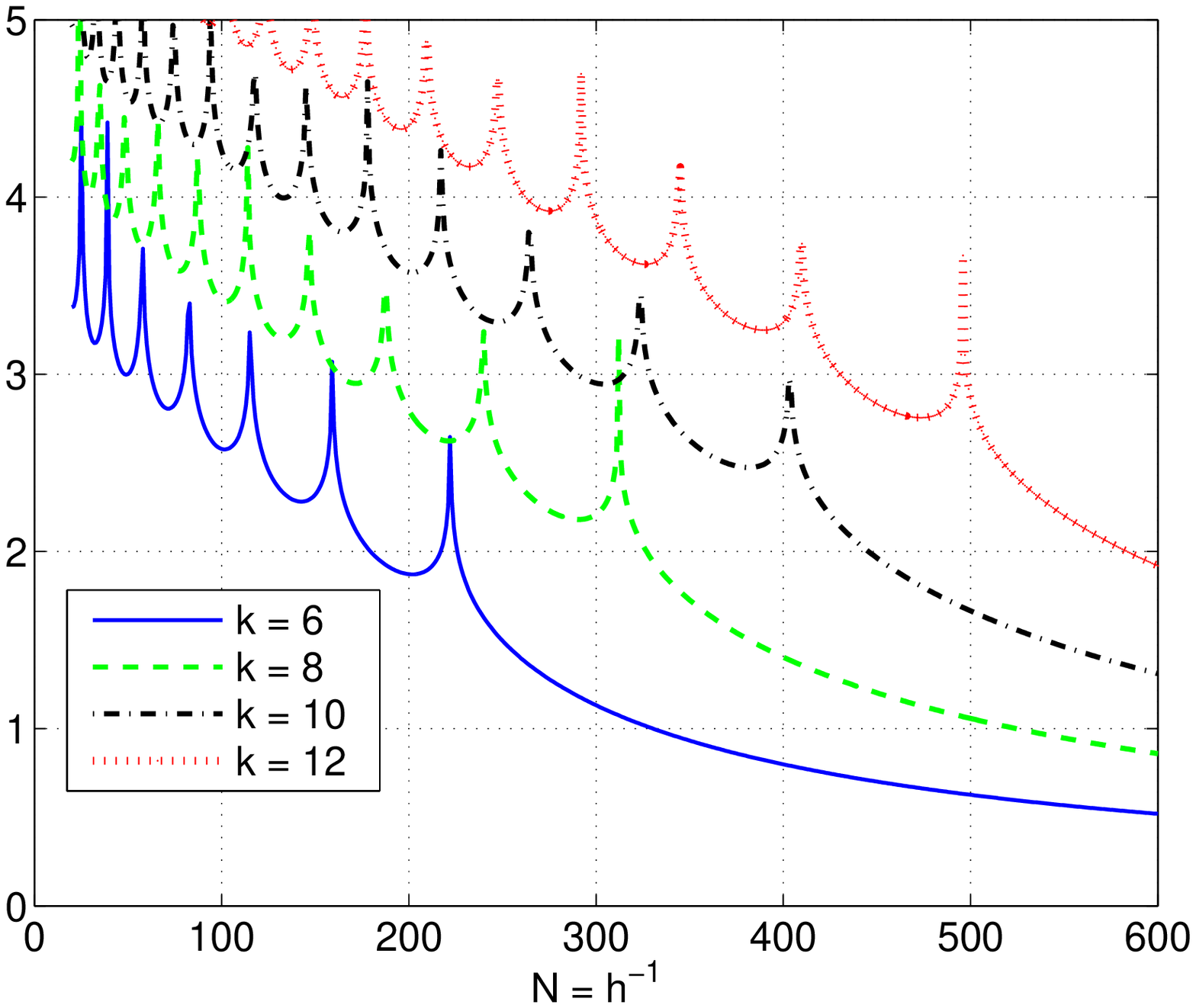}}
\caption{Plot of the function $\widetilde{q}_{k}(h)$ in (\protect\ref{qkk_2})
for $\protect\alpha =1/2,$ $\protect\tau =1/10$ and different values of $k.$}
\label{cons2_2}
\end{figure}
We conclude this section by considering what happens with the general
assumption $\left\vert y(t)\right\vert \leq M$. Using this bound, 
by (\ref{tm_2}) we consider the function
\begin{equation}
\overline{q}_{k}(h)=\log _{h}\left( h^{-\alpha
}\sum\nolimits_{j=0}^{N}\left\vert (-1)^{j}\binom{\alpha }{j}-\gamma
_{j,k}\right\vert \right) ,  \label{qkkk}
\end{equation}%
whose behavior is reported in Figure \ref{cons3_2}.\\

As already mentioned, the numerical analysis reported in this section can
also be used to select a proper value for $k$ for a fixed time stepping $h$\
or viceversa. Figures \ref{cons1_2} and \ref{cons3_2} are in fact independent of
the problem and can  be used easily to this aim.

\begin{center}
\begin{figure}
\centerline{\includegraphics[width=8cm,height=5cm]{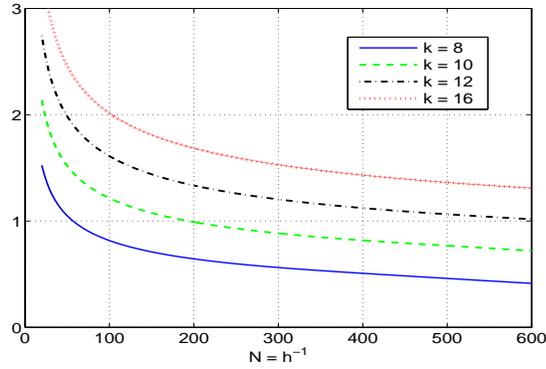}} \vspace{-0.4cm}
\caption{Plot of the function $\bar{q}_{k}(h)$ in (\protect\ref{qkkk}),
for $\protect\alpha =1/2,$ $\protect\tau =1/10$ and different values of $k.$}
\label{cons3_2}
\end{figure}
\end{center}

\subsection{Linear stability}

For what concerns the linear stability, taking $g(t,y(t))=\lambda y(t)$ in (%
\ref{fde}), we have that $y(t)=E_{\alpha }(\lambda t^{\alpha })\rightarrow 0$
for%
\begin{equation*}
\left\vert \arg (\lambda )-\pi \right\vert <\left( 1-\frac{\alpha }{2}%
\right) \pi ,
\end{equation*}%
(see (\ref{mlf}) and \cite{Lub2}). The absolute stability region of a FBDF is given by the
complement of $\left\{ \omega _{p}^{(\alpha )}(\zeta ):\left\vert \zeta
\right\vert \leq 1\right\} $ 
so that a good approximation of the generating function should lead to
similar stability domains and hence good stability properties. We consider
the behavior of methods based on the Gauss-Jacobi rule whose corresponding
stability regions are given by, see (\ref{pm})-(\ref{qm}),
\begin{equation*}
\mathbb{C}\backslash \left\{ \frac{p_{m}(\zeta )}{q_{m}(\zeta )}:\left\vert
\zeta \right\vert \leq 1\right\} .
\end{equation*}%
From a theoretical point of view, from Theorem~\ref{tipipc} one deduces that
for $p=1$ such regions are always unbounded for each $m=k$ and $\tau \in
(0,1].$ Indeed, as shown in Figure \ref{fig11}, the methods simulate the
behavior of the FBDFs rapidly, i.e. already for $k$ and therefore $m$ small.
In particular, the stability domain of the method of degree $k=m=12$ in the
left frame of Figure \ref{fig11} is very close to the one of the FBDF of
order $1.$
\begin{figure}[tbh]
\centerline{\includegraphics[width=12cm,height=6cm]{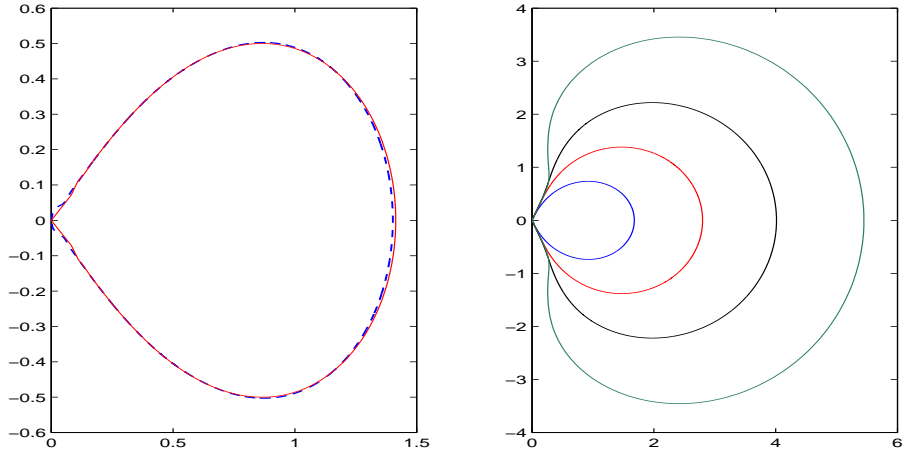}} \vspace{%
-0.4cm}
\caption{Left: boundary of the stability domains of the methods based on the
Gauss-Jacobi rule with $k=6$ (dashed line) and $k=12$ (solid line) for $p=1,$
$\protect\tau =1/10,$ and $\protect\alpha =1/2$. Right: boundary of the
stability domains of the same methods of degree $k=8$ for $p=1$ (inner) to $%
4 $ (outer), $\protect\tau =1/10$ and $\protect\alpha =3/4$.}
\label{fig11}
\end{figure}

\section{Numerical examples}

\label{sec6}

As first example, we consider the one-dimensional Nigmatullin's type
equation
\begin{eqnarray*}
_{0}D_{t}^{\alpha }u(x,t) &=&\frac{\partial ^{2}u(x,t)}{\partial x^{2}}%
,\quad t>0,\quad x\in \left( 0,\pi \right) , \\
u(0,t) &=&u(\pi ,t)=0, \\
u(x,0) &=&\sin x.
\end{eqnarray*}%
If we discretize the spatial derivative by applying the classical central
differences on a uniform mesh of meshsize $\delta =\pi /(s+1),$ we obtain
the $s$-dimensional FDE
\begin{equation}
_{0}D_{t}^{\alpha }y(t)=Ly(t),\quad y(0)=y_{0},  \label{df}
\end{equation}%
where $L=\delta ^{-2}\cdot \mbox{tridiag}(1,-2,1)$, and $y_{0}$ is the sine
function evaluated at the interior grid points. It is known that $y_{0}$ is
the eigenvector of $L$ corresponding to its largest eigenvalue $\lambda
=-4\sin ^{2}(\delta /2)/\delta ^{2}.$ This implies that the exact solution
of (\ref{df}) is given by, see (\ref{mlf}),
\begin{equation*}
y(t)=E_{\alpha }(t^{\alpha }\lambda )y_{0}.
\end{equation*}%
In Figure \ref{nigm} some results are reported. We compare the maximum norm
of the error at each step of the FBDF of order 1 (FBDF1) and the method
based on the Gauss-Jacobi rule for some values of $k$ and $\alpha .$ The
initial values for the $k$-step schemes are defined according to the
strategy described in Section~\ref{sec5}. The reference solutions have been
computed using the already mentioned \texttt{Matlab} function \texttt{mlf} from \cite{Pod2}.
The dimension of the problem is $s=50$, and we consider a uniform
time step $h=1/N$ with $N=250$ so that $h\approx \delta ^{2}$. As one can
see, if we set $\tau =1$, i.e. if we use the classical Pad\`{e}
approximation of $(1-\zeta )^{\alpha -1}$ (see Remark~\ref{rempade}), the $k$%
-step methods simulate quite well the FBDF1 initially and an improvement of
the results can be obtained by slightly increasing (considering to the total
number of integration steps) the stepnumber $k.$ A noticeable improvement
can be obtained by choosing a different value of $\tau .$ In particular, if
we set $\tau ={\hat{\tau}}=4k/N$ (see (\ref{tauhat})) then the $6$-step
method provides a numerical solution with the same accuracy of the one
provided by the FBDF1 over the entire integration interval.\newline

\begin{figure}[tbh]
\centerline{\includegraphics[width=12cm]{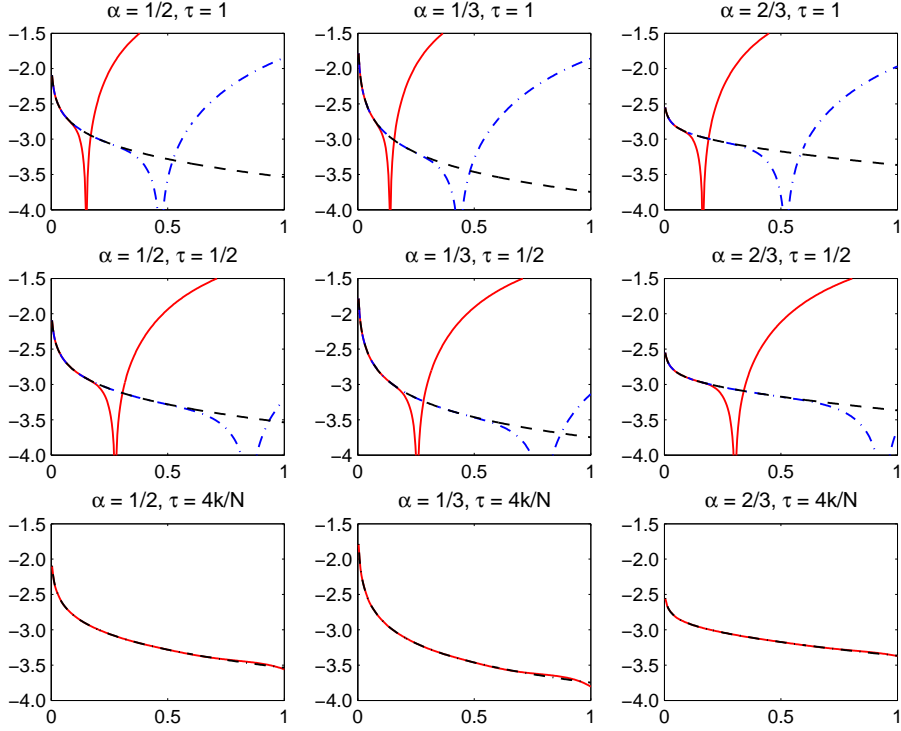}}
\caption{Step by step error (in logarithmic scale) for the solution of (%
\protect\ref{df}) for the FBDF of order 1 (dashed line) and the method based
on the Gauss-Jacobi rule with $k=6$ (solid line) and $k=12$ (dash-dotted
line).}
\label{nigm}
\end{figure}

As second example we consider the following nonlinear problem
\begin{eqnarray*}
_{0}D_{t}^{\alpha }u(x,t) &=&\frac{\partial (p(x)u(x,t))}{\partial x}%
+K_{\alpha }\frac{\partial ^{2}u(x,t)}{\partial x^{2}}+ru(x,t)\left( 1-\frac{%
u(x,t)}{K}\right) , \\
u(0,t) &=&u(5,t)=0,\quad t\in \lbrack 0,1], \\
u(x,0) &=&x^{2}(5-x)^{2},\quad x\in \lbrack 0,5].
\end{eqnarray*}%
This is a particular instance of the time fractional Fokker-Planck equation
with a nonlinear source term \cite{YLT}. In population biology, its solution
$u(x,t)$ represents the population density at location $x$ and time $t$ and
the nonlinear source term in the equation is known as Fisher's growth term.

The application of the classical second order semi-discretization in space
with stepsize $\delta =5/(s+1)$ leads to the following initial value problem
\begin{equation}
_{0}D_{t}^{\alpha }y(t)=Jy(t)+g(y(t)),\quad t\in (0,1],\qquad y(0)=y_{0},
\label{fpsemi}
\end{equation}%
where, for each $i=1,\ldots ,s,$ $(y(t))_{i}\equiv y_{i}(t)\approx u(i\delta
,t),$ $y_{i}(0)=u(i\delta ,0),$ $(g(y))_{i}=ry_{i}(1-y_{i}/K),$ and $J$ is a
tridiagonal matrix whose significant entries are
\begin{eqnarray*}
J_{ii} &=&p^{\prime }(x_{i})-\frac{2K_{\alpha }}{\delta ^{2}},\qquad
i=1,\ldots ,s, \\
J_{i,i-1} &=&-\frac{p(x_{i})}{2\delta }+\frac{K_{\alpha }}{\delta ^{2}}%
,\quad J_{i-1,i}=\frac{p(x_{i-1})}{2\delta }+\frac{K_{\alpha }}{\delta ^{2}}%
,\qquad i=1,\ldots ,s-1.
\end{eqnarray*}%
In our experiment, we set $\alpha =0.8,$ $p(x)=-1,$ $r=0.2,$ $K_{\alpha
}=K=1 $ (see \cite[Example 5.4]{YLT}) and $s=90.$ We solved (\ref{fpsemi})
over a uniform meshgrid with stepsize $h=1/256$ by using the FBDF1 and the $%
6 $-step method with $\tau =24/256.$ The so-obtained numerical solutions
have the same qualitative behavior as shown in Figure~\ref{fp_istant} for
different times $t=1/8,1/2,1.$ This is confirmed by the step by step maximum
norm of the difference between them reported in Figure~\ref{fp_norma}.

\begin{figure}[tbh]
\centerline{\includegraphics[width=12cm,height=6cm]{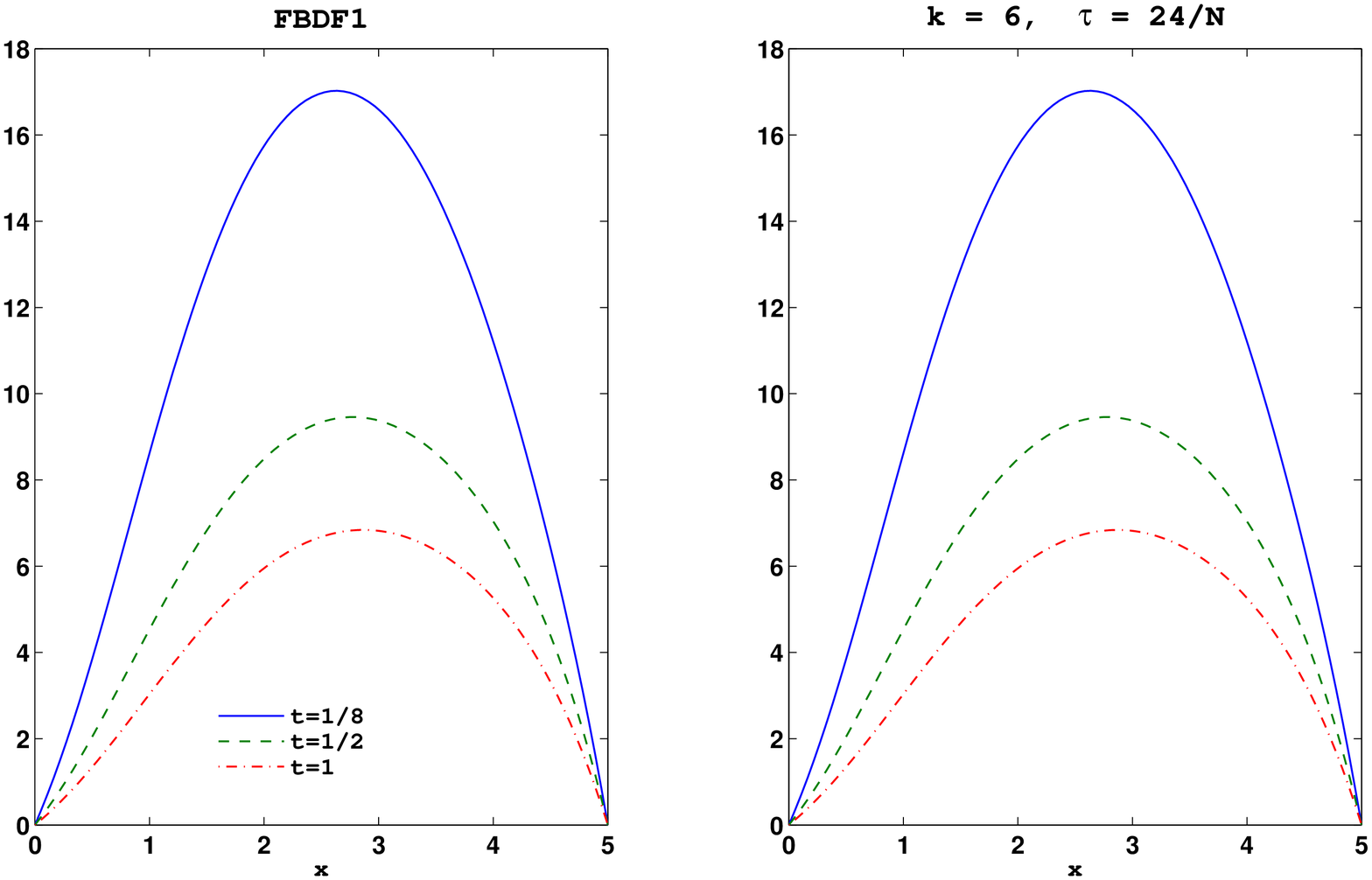}}
\caption{Numerical solution of (\protect\ref{fpsemi}) with $\protect\alpha %
=0.8$ provided by the FBDF1 and the method based on the Gauss-Jacobi rule
with $k=6$ at $t=1/8,1/2,1.$}
\label{fp_istant}
\end{figure}

\begin{figure}[htb]
\centerline{\includegraphics[width=7cm,height=5cm]{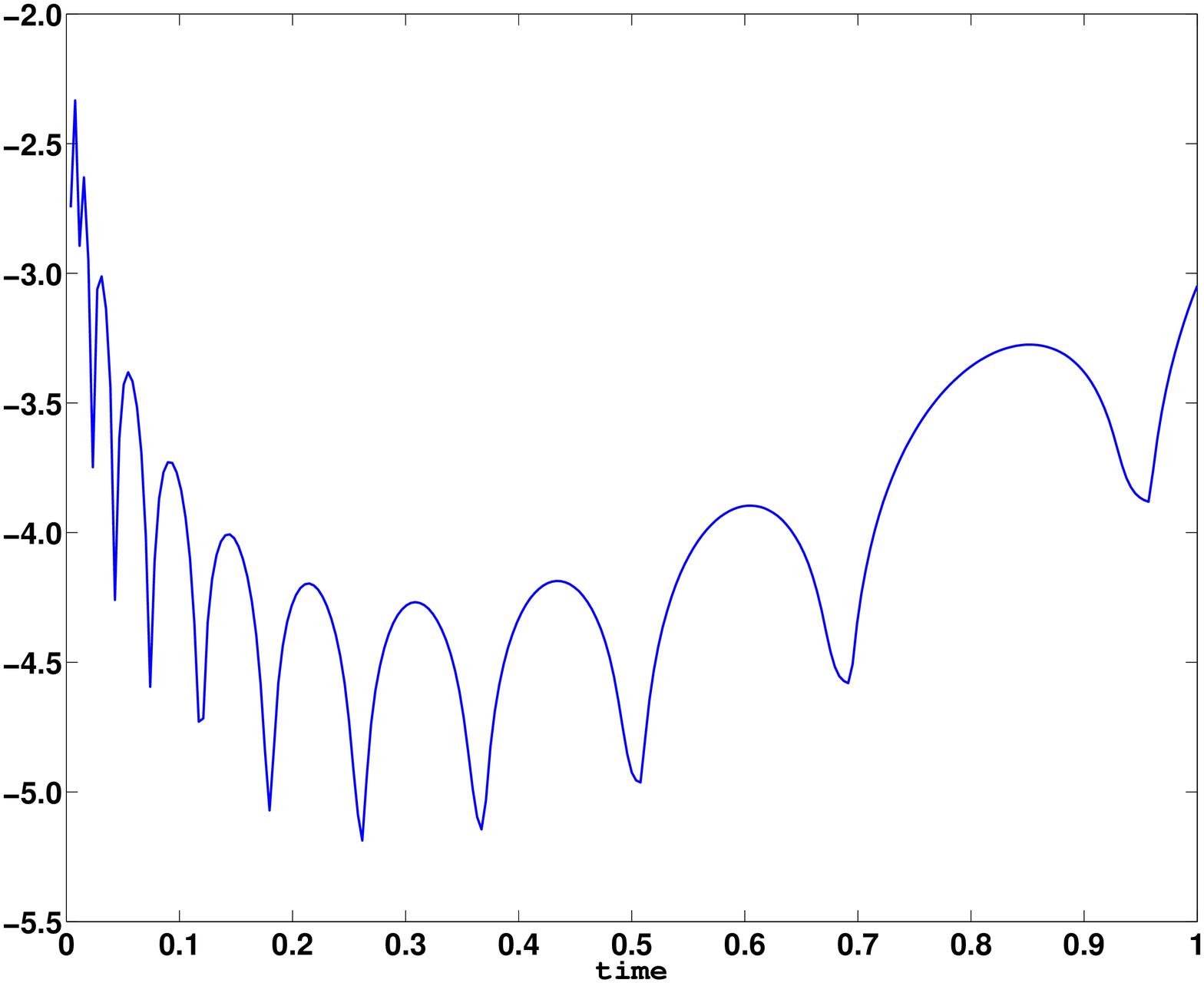}}
\caption{Step by step norm of the difference (in logarithmic scale) between
the numerical solutions provided by the FBDF1 and the $6$-step methods.}
\label{fp_norma}
\end{figure}

\section{Conclusion}

In this paper we have presented a new approach for the construction of $m$%
-step formulas for the solution of FDEs. The method shows encouraging
results in the discrete approximation of the FDE solution especially if we
consider the computational saving with respect to the attainable accuracy.
Indeed good results are attainable with short memory length. Theoretically
the method is $0$-stable and the consistency is well simulated. The linear
stability is preserved.

We finally remark that even if the paper only deals with the approximation
of FBDFs, the ideas can easily be extended to other approaches such as the
Fractional Adams type methods. It is just necessary to detect the generating
function or the corresponding Toeplitz matrix and then apply the technique
presented in the paper.

\section*{Acknowledgements}
This work was supported by the GNCS-INdAM 2014 project ``Metodi numerici per
modelli di propagazione di onde elettromagnetiche in tessuti biologici''.

\end{document}